\documentclass{amsart}
\usepackage[utf8]{inputenc}
\usepackage{amsthm}
\usepackage{ragged2e}
\usepackage[english]{babel}

\usepackage{amssymb}
\usepackage[utf8]{inputenc}
\usepackage{hyperref}
\hypersetup{
    colorlinks=true,
    linkcolor=blue,
    citecolor=blue,
    urlcolor=blue,
    pdfauthor={Juan Flores Torres},
    pdftitle={Moduli space of bi-invariant metrics}
}
\usepackage{tikz-cd}

\newtheorem{Thm}{Theorem}[subsection]

\newtheorem{Def}[Thm]{Definition}

\newtheorem{Prop}[Thm]{Proposition}

\newtheorem{Lema}[Thm]{Lemma}

\newtheorem{Coro}[Thm]{Corollary}

\newtheorem{Rem}[Thm]{Remark}

\newtheorem{SThm}[Thm]{Splitting Theorem}

\newtheorem{Wthm}[Thm]{Theorem of Wallach}

\newtheorem{Mtheorem}[Thm]{Theorem}

\title{Moduli Space of Bi-Invariant Metrics}
\author[J. Flores Torres]{Juan Flores Torres}
\address[J. Flores Torres]{Instituto de Matemáticas UNAM, México}
\email{juan.flores@ciencias.unam.mx}
\thanks{Supported by CONACyT scholarship No. 848752.}

\date{\today}


\begin{document}

\maketitle

\begin{abstract}
In this work, we focus on describing  the space of bi-invariant metrics in a Lie group up to isometry. I.e, that is, metrics invariant under both left and right translations. We show that $\mathfrak{BI}$, the moduli space of bi-invariant metrics, is an orbifold. Moreover we give an explicit description of this orbifold, and of $\mathfrak{EBI}$, the space of bi-invariant metrics equivalent under isometries and scalar multiplies.
\end{abstract}

\section{Introduction}


In the study of Riemannian manifolds, it is of interest to show how complicated is the collection of Riemannian geometries satisfying a certain property. For example, given a manifold $M$ we could explore the class of all Riemannian geometries defined on $M$, such that the curvature induced by these geometries has a fixed lower bound; we could also consider the class of all  geometries that turn $M$ into an Einstein manifold. Moreover, one can ask if given two geometries with the desired property, we can deform one into the other via geometries satisfying the given property. This can be quantified in a precise way, by considering the \emph{moduli space} of all Riemannian metrics satisfying the desired property. This set consist of all the Riemannian metrics up to isometry, equipped with the $C^\infty$-Whitney topology (see \cite{TuschmannWraith}).

For a Lie group $G$,  a natural interesting  family of  Riemannian geometries consist of the ones that have $G$ contained in their isometry group, i.e. left or right invariant Riemannian metrics. A smaller subfamily  consist on the Riemannian geometries that are both left and right invariant, that is \emph{bi-invariant} metrics. These Riemannian metrics satisfy non-negative lower curvature bounds (see \cite{Milnor}), and have important applications when studying spaces with a lower curvature bound (see \cite{Searle}).

In this work we study the topology of the moduli space of bi-invariant metrics of a compact Lie group, obtaining the following theorem. 

\begin{Mtheorem}\label{Teorema Intro}
Consider a compact Lie group $G$ with Lie algebra $\mathfrak{g} =\mathfrak{a}\oplus \mathfrak{s}\oplus \mathfrak{b}_1\oplus \cdots\oplus \mathfrak{b}_l$, where $\mathfrak{s}$, $\mathfrak{b}_1,\cdots \mathfrak{b}_l$,  are semi-simple subalgebras, such that any two factors of  $\mathfrak{s}$  are non isomorphic to each other, each $\mathfrak{b}_i$ decomposes as the sum of $m_{i}$ isomorphic factors, no simple factor of $\mathfrak{s}$ is isomorphic to any simple factor of any $\mathfrak{b}_{i}$, and $\mathfrak{a}$ is an abelian subalgebra. Then the moduli space of bi-invairant metrics $\mathfrak{BI}(\mathfrak{g})$ is a contractible orbifold homeomorphic to 
$$\widetilde{\mathfrak{BI}}(\mathfrak{s})\times SP^{m_{1}}(\mathbb{R})\times\cdots\times SP^{m_{l}}(\mathbb{R}).$$
Here $\widetilde{\mathfrak{BI}}(\mathfrak{s})$ is the space of all possible bi-invariant metrics on $\mathfrak{s}$, and $SP^{m_{i}}(\mathbb{R})$ is the $m_i$th symmetric product of $\mathbb{R}$.
\end{Mtheorem}


As mentioned before bi-invariant metrics are useful in the construction of manifolds with lower curvature bounds.  For example given a Lie group acting by isometries, Cheeger \cite{Cheeger} gave a deformation procedure that will preserve non-negative curvature bounds, by fixing a bi-invariant metric. With this deformation one can construct new spaces with positive Ricci curvature \cite{SearleWilhelm}. Moreover, for closed manifolds of cohomogeneity one, Grove and Ziller used a fixed bi-invariant metric to show the existence of invariant Riemannian metrics with non-negative curvature \cite{GroveZiller}.

Thus it is relevant to study how bi-invariant metrics relate to each other. Due to the strong geometric properties of these metrics, it  is natural that the homotopy information of the moduli space of bi-invariant metrics is simple. This is in contrast to the moduli-space of left invariant metrics as observed by Kodama, Takahara and Tamaru \cite{Kodama}.

Nonetheless, we believe it is useful to have a full description of the topology of this moduli space. In particular observe that the conclusion in Theorem~\ref{Teorema Intro} states that the moduli space of bi-invariant metrics of a compact Lie group is an orbifold, and moreover, we can read the orbifold stratification from the Lie algebra decomposition. 


It is important to note that the characterization of Theorem \ref{Teorema Intro} only relies on the group admitting a bi-invariant metric and the decomposition of the semi-simple factor in the Lie algebra, not on the compactness of the Lie group. The compactness stated in Theorem \ref{Teorema Intro} is stated to guarantee the existence of a bi-invariant metric. We recall that there are examples of non-compact and non-abelian simple Lie groups, such as $SL(2,\mathbb{C})$, that do not admit bi-invariant metrics.




We start in Section \ref{Section 1} by collecting results that characterize Lie groups and Lie algebras that possess bi-invariant metrics. Specifically, we discover that the decomposition of these Lie algebras yields comprehensive insights into the behavior of the automorphism group on the space of bi-invariant metrics. This enables us to describe the space of bi-invariant metrics in Section \ref{Section 3}. We do this using the factorization of the Lie algebra and a factorization of the action of the automorphism group of the Lie algebra. Likewise, we employ the same procedure to obtain the space of equivalent bi-invariant metrics, encompassing scalar multiples as well. In Section \ref{Section 4} we present explicitly the moduli spaces of bi-invariant metrics for Lie groups of dimension at most $6$. We end by restating in Section \ref{Section 5} some results from \cite{Milnor}.

\section*{Acknowledgments}

I would like to express my heartfelt gratitude to Dr. Diego Corro for his invaluable support and guidance in this research. I also want to thank my family for their unconditional support. I dedicate this work to the memory of my father, Juan Flores Fuentes, who guided me through my early abstract processes. Without their support, this research would not have been possible.

\section{The Geometry of a Lie Group}\label{Section 1}

A left-invariant metric on a Lie group $G$ is a Riemannian metric that is invariant under left translations. It is easy to see that left-invariant metrics on $G$ correspond to inner products on the Lie algebra $\mathfrak{g}$. Thus, the space $\widetilde{\mathfrak{M}}(\mathfrak{g})$ (defined in \cite{Kodama} by Kodama, Takahara and Tamaru) consisting of inner products on the Lie algebra $\mathfrak{g}$ corresponds to the space of left-invariant metrics on $G$. 

\subsection{Bi-invariant metrics}

A Riemannian metric on $G$ is called bi-invariant if it is
invariant under both left and right translation. 
\\

We begin by recalling some facts about the adjoint representation. Each $g\in G$ defines an inner automorphism $\psi_{g}:G\rightarrow G$, $\psi_{g}(x)=gxg^{-1}$. Using this, we can define a group homomorphism $Ad:G\rightarrow Aut(\mathfrak{g})$, called the adjoint representation of $G$, where $Ad(g):\mathfrak{g}\rightarrow\mathfrak{g}$ is the differential of $\psi_{g}$ at the identity element of $G$. The differential of $Ad:G\rightarrow Aut(\mathfrak{g})$ gives us a representation between the Lie algebras $ad:\mathfrak{g}\rightarrow Der(\mathfrak{g})$, and thus we obtain the following commutative diagram.

$$
\begin{tikzcd}
\mathfrak{g}\arrow{r}{ad} \arrow[swap]{d}{exp} & Der(\mathfrak{g}) \arrow{d}{\mathrm{e}} \\%
		G \arrow{r}{Ad}& Aut(\mathfrak{g}).
		\end{tikzcd}
$$

Where ``$\mathrm{e}$'' denotes the exponential map of $Aut(\mathfrak{g})$. Next, we state a couple results regarding bi-invariant metrics that can be found in \cite{Milnor}.

\begin{Lema}\label{adinbi}

A left-invariant metric in a Lie group $G$ is right-invariant if and only if $Ad(g):\mathfrak{g}\rightarrow\mathfrak{g}$ is an isometry for any element $g\in G$.
    
\end{Lema}

\begin{Lema}\label{skebi}

Let $G$ a connected Lie group, a left invariant metric $\langle\cdot,\cdot\rangle$ is bi-invariant if and only if $ad(x)$ is skew-adjoint for all $x\in\mathfrak{g}$. That is, $$\langle ad(x)(y),(z)\rangle=-\langle y,ad(x)(z)\rangle,$$ for all $x,y,z\in\mathfrak{g}$.

\end{Lema}

\begin{Rem}

Lemma \ref{skebi} allows us to characterize a bi-invariant metric for a Lie algebra: a metric $\langle\cdot,\cdot\rangle\in\widetilde{\mathfrak{M}}(\mathfrak{g})$ is bi-invariant if and only if $ad(x)$ is skew-adjoint for all $x\in\mathfrak{g}.$
    
\end{Rem}

For compact groups we have the following theorem that can be found in \cite{Alexandrino}.

\begin{Thm}\label{bicom}

Any compact Lie group G admits a bi-invariant metric.

\end{Thm}

For a Lie algebra with a bi-invariant metric we have the following splitting theorem that can be found in \cite{Milnor}.
 
\begin{SThm}\label{split}

Let $\mathfrak{g}$ be a Lie algebra with a bi-invariant metric. Then we have $\mathfrak{g}=\mathfrak{a}_{1}\oplus\cdots\oplus\mathfrak{a}_{k}$ an orthogonal direct sum of simple ideals and commutative ideals without proper ideals, where the simply connected Lie group $\widetilde{G}$ associated with $\mathfrak{g}$ can be expressed as the product $A_{1}\times\cdots\times A_{k}$ of normal subgroups. Furthermore, for each $A_{i}$ we have two options:

\begin{enumerate}
\item If $\mathfrak{a}_{i}$ is commutative, then it has dimension $1$ and $A_{i}\cong\mathbb{R}.$
\item If $\mathfrak{a}_{i}$ is non-commutative, then the center of $\mathfrak{a}_{i}$ must be trivial, $A_{i}$ has strictly positive Ricci curvature and $A_{i}$ is compact.
\end{enumerate}
    
\end{SThm}

\begin{Rem}\label{lialgbi}

With the Splitting Theorem \ref{split}, we can say that any Lie algebra with a bi-invariant metric decomposes as $\mathfrak{g}=\mathfrak{s}\oplus Z(\mathfrak{g})$, where $\mathfrak{s}$ is a semi-simple Lie algebra.
    
\end{Rem}

For a connected Lie group that admits a bi-invariant metric, Milnor \cite{Milnor} gives the following description.

\begin{Lema}\label{bi5}

A connected Lie group admits a bi-invariant metric if and only if it is isomorphic as a group to the cartesian product $K\times\mathbb{R}^{l}$, where $K$ is compact.
    
\end{Lema}

\subsection{Curvature of Bi-invariant metrics}

Given a left-invariant metric on $G$ we have its associated Levi-Civita connection. We consider the curvature tensor $R$ of the left invariant metric, and define the following curvature operator: for $x,y\in\mathfrak{g}$ we define $$\kappa(x,y)=\langle R_{xy}(x),y \rangle.$$ For a bi-invariant metric, Milnor \cite{Milnor} shows that this curvature operator is non-negative.

\begin{Thm}\label{curvsecbicor}
    
In a Lie group $G$ with a bi-invariant metric, for $x,y\in\mathfrak{g}$ then $\kappa(x,y)=\frac{1}{4}||[x,y]||^{2}\geq0.$

\end{Thm}

 All curvatures we consider are determined by this operator $\kappa.$

\section{Moduli Space of Bi-Invariant Metrics}\label{Section 3}

In this section, we examine the moduli space of bi-invariant metrics of a Lie
group $G$. An advantage of this approach is the fact that to find left invariant metrics
with non-negative sectional and scalar curvatures on a Lie group $G$, under certain conditions (see \cite[p.297]{Milnor}), it will be sufficient to find a subgroup with
a bi-invariant metric. In this way, knowing how many non isomorphic bi-invariant
metrics the subgroup has, we have a rough measure of how many, non isomorphic left
invariant metrics with non-negative sectional and scalar curvatures the group $G$ possesses.

\subsection{Isometry Classes of Bi-Invariant Metrics}

We denote the collection of inner products on a Lie algebra $\mathfrak{g}$ by $\widetilde{\mathfrak{M}}(\mathfrak{g})$. Identifying an inner product on $\widetilde{\mathfrak{M}}(\mathfrak{g})$ with a matrix, we endow $\widetilde{\mathfrak{M}}(\mathfrak{g})$ with the relative topology as a subspace of a matrix space.

\begin{Def}

For a Lie algebra $\mathfrak{g}$, we define the set of left invariant metrics that are bi-invariant in $\mathfrak{g}$, as follows

$$\widetilde{\mathfrak{BI}}(\mathfrak{g})=\{\langle\cdot ,\cdot \rangle\in\widetilde{\mathfrak{M}}(\mathfrak{g})|\langle \cdot,\cdot \rangle \ is \ a \ bi-invariant \ metric\}.$$

\end{Def}

We want to study the \textit{moduli space of bi-invariant metrics} $\mathfrak{BI}(\mathfrak{g})$ consisting of the isometry classes of $\widetilde{\mathfrak{BI}}(\mathfrak{g})$ and the space $\mathfrak{EBI}(\mathfrak{g})$ of conformally equivalent classes  $\widetilde{\mathfrak{BI}}(\mathfrak{g})$. To do this we introduce the following definition.

\begin{Def}

Let $(\mathfrak{g}_{1},\langle\cdot,\cdot\rangle_{1})$ and $(\mathfrak{g}_{2},\langle\cdot,\cdot\rangle_{2})$ Lie algebras with inner products. We say that:

\begin{flushleft}

\begin{enumerate}
    \item They are isometric if there exists a Lie algebra isomorphism  $\phi:\mathfrak{g}_{1}\rightarrow\mathfrak{g}_{2}$ such that $$\phi^{*}(\langle\cdot,\cdot\rangle_{2})=\langle\cdot,\cdot\rangle_{1}. $$

    \item They are conformally equivalent, if there exists a Lie algebra isomorphism   $\phi:\mathfrak{g}_{1}\rightarrow\mathfrak{g}_{2}$ and a real number $\lambda>0$ such that $$\lambda\phi^{*}(\langle\cdot,\cdot\rangle_{2})=\langle\cdot,\cdot\rangle_{1}. $$

\end{enumerate}
\end{flushleft}
    
\end{Def}

It is clear that both are equivalence relations on $\widetilde{\mathfrak{BI}}(\mathfrak{g})$. Now we are going to prove that an equivalence class of a bi-invariant metric contains only bi-invariant metrics.

\begin{Prop}

Let $\mathfrak{g}$ a Lie algebra and $\langle\cdot,\cdot\rangle_{0}\in\widetilde{\mathfrak{BI}}(\mathfrak{g})$, then for any Lie algebra isomorphism $\phi:\mathfrak{g}\rightarrow\mathfrak{g}$ and $\lambda>0$ we have that $\phi^{*}(\langle\cdot,\cdot\rangle_{0})$ 
and $\lambda\phi^{*}(\langle\cdot,\cdot\rangle_{0})$ are bi-invariant metrics.

\end{Prop}

\begin{proof}

Using \ref{skebi} we know that $ad(x)$ is skew-adjoint for all $x\in\mathfrak{g}$. That is to say $$\langle [x,y],z\rangle_{0}=-\langle y,[x,z] \rangle_{0}.$$

Now for $\phi\in Aut(\mathfrak{g})$ we have $$\langle \phi([x,y]),\phi(z)\rangle_{0}=\langle[\phi(x),\phi(y)] ,\phi(z) \rangle_{0}=-\langle \phi(y),[\phi(x),\phi(z)] \rangle_{0}. $$

This implies that $ad$ is skew-adjoint with respect to $\phi^{*}(\langle\cdot,\cdot\rangle_{0})$ and by \ref{skebi} we have that $\phi^{*}(\langle\cdot,\cdot\rangle_{0})$ 
and $\lambda\phi^{*}(\langle\cdot,\cdot\rangle_{0})$ are bi-invariant metrics.

\end{proof}

Now we introduce $\mathfrak{BI}(\mathfrak{g})$ and $\mathfrak{EBI}(\mathfrak{g})$.

\begin{Def}

For a Lie algebra $\mathfrak{g}$ with a bi-invariant metric we define:

\begin{enumerate}
\item $\mathfrak{BI}(\mathfrak{g})$ as the space of isometry classes of  $\widetilde{\mathfrak{BI}}(\mathfrak{g})$ 

\item $\mathfrak{EBI}(\mathfrak{g})$ as the space of conformally equivalent classes of $\widetilde{\mathfrak{BI}}(\mathfrak{g})$
\end{enumerate}
    
\end{Def}

We conclude from the definition of both equivalence relations the following corollary.

\begin{Coro}

For a Lie algebra $\mathfrak{g}$ with a bi-invariant metric, we have:

\begin{enumerate}

\item $\mathfrak{BI}(\mathfrak{g})$ is the orbit space $\widetilde{\mathfrak{BI}}(\mathfrak{g})/Aut(\mathfrak{g})$ under the left action  $$(\phi,\langle\cdot ,\cdot \rangle)\mapsto\phi^{*}(\langle\cdot,\cdot\rangle).$$ 

\item $\mathfrak{EBI}(\mathfrak{g})$ is the orbit space $\widetilde{\mathfrak{BI}}(\mathfrak{g})/\mathbb{R}^{\times}Aut(\mathfrak{g})$ where $$\mathbb{R}^{\times}=\{\psi:\mathfrak{g}\rightarrow\mathfrak{g}|\psi(x)=\lambda x,\lambda>0\},$$
and $\mathbb{R}^{\times}Aut(\mathfrak{g})$ acts on $\widetilde{\mathfrak{BI}}(\mathfrak{g})$ via $(\lambda\phi,\langle\cdot,\cdot\rangle)\mapsto(\lambda\phi^{*}(\langle\cdot,\cdot\rangle))$.

\end{enumerate}

\end{Coro}

\begin{Rem}\label{acom}

It is easy to verify that $\lambda\phi^{*}(\langle\cdot,\cdot\rangle)=\phi^{*}(\lambda\langle\cdot,\cdot\rangle)$ wich is a crucial fact for the description of $\mathfrak{EBI}$.

\end{Rem}

The simplest case is when we have an abelian Lie algebra.

\begin{Coro}\label{biab}

Let $\mathfrak{g}$ an abelian Lie algebra, then we have that $\mathfrak{BI}(\mathfrak{g})=\{\langle\cdot,\cdot\rangle_{0}\}$  and $\mathfrak{EBI}(\mathfrak{g})=\{\langle\cdot,\cdot\rangle_{0}\}.$
    
\end{Coro}

\begin{proof}

In an abelian Lie algebra any metric is bi-invariant, in addition we have $GL(\mathfrak{g})=Aut(\mathfrak{g})$. This implies that two metrics in $\mathfrak{g}$ are always related under an automorphism, thus $\mathfrak{BI}(\mathfrak{g})=\{\langle\cdot,\cdot\rangle_{0}\}$  and $\mathfrak{EBI}(\mathfrak{g})=\{\langle\cdot,\cdot\rangle_{0}\}.$

\end{proof}

\subsection{Simplifying to the compact and semi-simple case}

Using the Splitting Theorem \ref{split} and Lemma \ref{semis} we can describe the group $Aut(\mathfrak{g})$ for a Lie algebra with a bi-invariant metric. From now on, every time we write $\mathfrak{g}=\mathfrak{s}\oplus Z(\mathfrak{g})$ we refer to a Lie algebra that admits a bi-invariant metric as indicated in Remark \ref{lialgbi}. Before starting the study of $Aut(\mathfrak{g})$, we are going to prove the following lemma.

\begin{Lema}\label{semis}

For a semi-simple Lie algebra $\mathfrak{g}$, then $[\mathfrak{g},\mathfrak{g}]=\mathfrak{g}.$
    
\end{Lema}

\begin{proof}

It is clear that $[\mathfrak{g},\mathfrak{g}]\subset\mathfrak{g}$. Since $\mathfrak{g}$ is semi-simple then $\mathfrak{g}=\mathfrak{a}_{1}\oplus\cdot\cdot\cdot\oplus\mathfrak{a}_{k}$ where each $\mathfrak{a}_{i}$ is simple, so $\mathfrak{a}_{i}$ is non abelian and contains no nonzero proper ideals. This implies that $[\mathfrak{a}_{i},\mathfrak{a}_{i}]=\mathfrak{a}_{i}$. Therefore given $x\in\mathfrak{g}$ we have $x=x_{1}+\cdot\cdot\cdot+x_{k}$ with $x_{i}\in\mathfrak{a}_{i}=[\mathfrak{a}_{i},\mathfrak{a}_{i}]$. Thus $x$ is is a linear combination of
elements of $[\mathfrak{g},\mathfrak{g}]$ and it follows that $[\mathfrak{g},\mathfrak{g}]=\mathfrak{g}$.   
    
\end{proof}

\begin{Prop}

For a Lie algebra with a bi-invariant metric $\mathfrak{g}\cong\mathfrak{s}\oplus Z(\mathfrak{g})$ we have $Aut(\mathfrak{g})\cong Aut(\mathfrak{s})\oplus Aut(Z(\mathfrak{g})).$

\end{Prop}

\begin{proof}

For $\phi\in Aut(\mathfrak{g})$, we will see that $\phi(\mathfrak{s})\subset\mathfrak{s}$ and $\phi(Z(\mathfrak{g}))\subset Z(\mathfrak{g})$. Take $x\in\mathfrak{s}$, so by Lemma \ref{semis} we know that $x=\sum\alpha_{i}[u_{i},v_{i}]$ for a finite number of $v_{i},u_{i}\in\mathfrak{s}.$ Then we have $\phi(x)=\phi(\sum\alpha_{i}[u_{i},v_{i}])=\sum\alpha_{i}[\phi(v_{i}),\phi(u_{i})]\in\mathfrak{s}$, since $Z(\mathfrak{g})$ is abelian and thus the Lie bracket in $\mathfrak{s}\oplus Z(\mathfrak{g})$ has
nontrivial components in $\mathfrak{s}$. Finally if $y\in Z(\mathfrak{g})$ and $w\in\mathfrak{g}$ we have $[\phi(y),w]=\phi([y,\phi^{-1}(w)])=0$ therefore $\phi(y)\in Z(\mathfrak{g})$ as claimed.
    
\end{proof}

\begin{Lema}\label{autbi}

For the Lie algebra $\mathfrak{g}\cong\mathfrak{s}\oplus Z(\mathfrak{g})$, the spaces $\mathfrak{BI}(\mathfrak{g})$ and $\mathfrak{EBI}(\mathfrak{g})$ are homeomorphic to $\mathfrak{BI}(\mathfrak{s})$ and $\mathfrak{EBI}(\mathfrak{s})$. 
    
\end{Lema}

\begin{proof}

Let us consider a metric $\langle\cdot,\cdot\rangle_{\mathfrak{g}}$ on $\mathfrak{g}\cong\mathfrak{s}\oplus Z(\mathfrak{g})$. This metric can be expressed as a sum of metrics $\langle\cdot,\cdot\rangle_{\mathfrak{g}}=\langle\cdot,\cdot\rangle_{\mathfrak{s}}+\langle\cdot,\cdot\rangle_{Z(\mathfrak{g})}.$  If $\lambda(\phi\oplus\psi)\in\mathbb{R}^{\times} (Aut(\mathfrak{s})\oplus Aut(Z(\mathfrak{g})))$ then

$$\lambda(\phi\oplus\psi)^{*}(\langle\cdot ,\cdot \rangle_{\mathfrak{s}}+\langle\cdot ,\cdot \rangle_{Z(\mathfrak{g})})=\lambda\phi^{*}(\langle\cdot ,\cdot \rangle_{\mathfrak{s}})+\lambda\psi^{*}(\langle\cdot,\cdot\rangle_{Z(\mathfrak{g})}).$$

Thus we can identify the moduli space $\widetilde{\mathfrak{BI}}(\mathfrak{s}\oplus Z(\mathfrak{g}))/\mathbb{R}^{\times}Aut(\mathfrak{s}\oplus Z(\mathfrak{g}))$ with the product $$(\widetilde{\mathfrak{BI}}(\mathfrak{s})/\mathbb{R}^{\times}Aut(\mathfrak{s}))\times(\widetilde{\mathfrak{BI}}(Z(\mathfrak{g}))/\mathbb{R}^{\times}Aut(Z(\mathfrak{g}))).$$ Since $Z(\mathfrak{g})$ is abelian, by Corollary \ref{biab} we have $\widetilde{\mathfrak{BI}}(Z(\mathfrak{g}))/\mathbb{R}^{\times}Aut(Z(\mathfrak{g}))=\{\langle\cdot,\cdot\rangle_{0}\}$, then $\mathfrak{EBI}(\mathfrak{g})=\mathfrak{EBI}(\mathfrak{s})\times\{\langle\cdot,\cdot\rangle_{0}\}$ wich is homeomorphic to $\mathfrak{EBI}(\mathfrak{s}).$ In conclusion we have that for $\mathfrak{g}\cong\mathfrak{s}\oplus Z(\mathfrak{g})$ the moduli spaces $\mathfrak{BI}(\mathfrak{g})$ and $\mathfrak{EBI}(\mathfrak{g})$ are homeomorphic to $\mathfrak{BI}(\mathfrak{s})$ and $\mathfrak{EBI}(\mathfrak{s})$ respectively.

\end{proof}

With this lemma, we can conclude that the study of moduli spaces $\mathfrak{BI}(\mathfrak{g})$ and $\mathfrak{EBI}(\mathfrak{g})$ reduces to the case of compact and semi-simple Lie groups. We know from Theorem \ref{bicom} that any compact Lie group admits a bi-invariant metric. The following lemma will be useful for examining what happens when the group is compact and simple.

\begin{Lema}\label{bisim}

If we have a compact and simple Lie group, then the bi-invariant metric is unique up to a positive scalar multiple.
    
\end{Lema}

The reader interested in a proof, may consult \cite{Alexandrino}.

\begin{Rem}\label{killbi}

In \cite[Theorem 2.35]{Alexandrino}, it is verified that the Killing form, $B(X, Y) = tz(ad(X)ad(Y))$ is $Ad$ invariant. Furthermore, if the group is compact, connected, and semi-simple, then $B$ is negative definite and $-B$ plays the role of a bi-invariant metric. In the case of Lemma \ref{bisim}, the unique metric up to scalar multiple is represented by the Killing form. That is, for any bi-invariant metric $\langle\cdot,\cdot\rangle$ on $\mathfrak{g}$ there exists $\alpha>0$ such that $-\alpha B = \langle\cdot,\cdot\rangle$.
    
\end{Rem}

\begin{Lema}\label{semisbior}

Let $\mathfrak{s}=\mathfrak{a}_{1}\oplus\cdot\cdot\cdot\oplus\mathfrak{a}_{k}$  a semi-simple Lie algebra, and $\langle\cdot,\cdot\rangle\in\widetilde{\mathfrak{BI}}(\mathfrak{s})$. Then, $\langle\cdot,\cdot\rangle$ restricted to each $\mathfrak{a}_{i}$ is bi-invariant in $\mathfrak{a}_{i}$, and for all $u\in\mathfrak{a}_{i}$, $v\in\mathfrak{a}_{j}$ with $i\neq j$, it holds that $\langle u,v\rangle=0$
    
\end{Lema}

\begin{proof}

Let's recall that the Lie bracket in $\mathfrak{s}$ decomposes as $[\cdot,\cdot]=[\cdot,\cdot]_{1}+\cdot\cdot\cdot+[\cdot,\cdot]_{k}$, where each $[\cdot,\cdot]_{i}$ is the bracket of the summand $\mathfrak{a}_{i}$. Now, let $x$, $y$, and $z$ be in $\mathfrak{a}_{i}$. Using Lemma \ref{skebi}, we have $\langle[x,y]_{i},z\rangle=\langle[x,y],z\rangle=-\langle y,[x,z]\rangle=-\langle y,[x,z]_{i}\rangle$. This implies that $\langle\cdot,\cdot\rangle$ restricted to $\mathfrak{a}_{i}$ is skew-adjoint, and by Lemma \ref{skebi}, it is bi-invariant in $\mathfrak{a}_{i}$.

Finally, let $u\in\mathfrak{a}_{i}$, $v\in\mathfrak{a}_{j}$ with $i\neq j$. Using Lemma \ref{semis}, since $\mathfrak{a}_{i}$ is simple, we have $u=\sum_{l}\alpha_{l}[u'_{l},u''_{l}]$ with $u'_{l},u''_{l}\in\mathfrak{a}_{i}$. With this in mind, we calculate $$\langle u,v\rangle=\sum_{l}\alpha_{l}\langle[u'_{l},u''_{l}],v\rangle=\sum_{l}-\alpha_{l}\langle u''_{l},[u'_{l},v]\rangle=0.$$

This is because $[u'_{l},v]=0$ since $u'_{l}\in\mathfrak{a}_{i}$ and $v\in\mathfrak{a}_{j}$.

\end{proof}

As previously mentioned in Lemma \ref{autbi}, the study of our moduli spaces is reduced to the case of compact, semi-simple groups. The following result gives us a description of the space $\mathfrak{BI}(\mathfrak{s})$, when $\mathfrak{s}$ semi-simple.

\begin{Thm}\label{bitilde}

If $S$ is a compact, connected, and semi-simple Lie group with Lie algebra decomposition $\mathfrak{s}=\mathfrak{a}_{1}\oplus\cdots\oplus\mathfrak{a}_{k}$, where $\mathfrak{a}_{i}$ is simple, it holds that $\widetilde{\mathfrak{BI}}(\mathfrak{s})$ can be identified with $\{(\alpha_{1},...,\alpha_{k})|\alpha_{i}>0 \}\subset\mathbb{R}^{k}$.

\end{Thm}

\begin{proof}

For $\mathfrak{s} = \mathfrak{a}_{1} \oplus \cdots \oplus \mathfrak{a}_{k}$, we know from Lemma \ref{bisim} that each $\mathfrak{a}_{i}$ admits a unique bi-invariant metric up to positive scalar multiples, and without loss of generality by Remark \ref{killbi} we can assume that it is given by the Killing form $-B_{i}$. In this way, for each $\alpha_{i} > 0$, the metric $$\alpha_{1}(-B_{1}) + \cdots + \alpha_{k}(-B_{k})$$ is bi-invariant in $\mathfrak{s}$. Now by Lemma \ref{semisbior}, if $\langle\cdot,\cdot \rangle_{\mathfrak{s}}$ is a bi-invariant metric in $\mathfrak{s}$, then the restriction of $\langle\cdot,\cdot \rangle_{\mathfrak{s}}$ to each $\mathfrak{a}_{i}$ is a bi-invariant metric in $\mathfrak{a}_{i}$. As a result, $\langle\cdot,\cdot \rangle_{\mathfrak{s}}$ decomposes uniquely as a bi-invariant metric in each summand, so we have that $$\langle\cdot,\cdot \rangle_{\mathfrak{s}} = \alpha_{1}(-B_{1}) + \cdots + \alpha_{k}(-B_{k}).$$

Thus, we can identify $\widetilde{\mathfrak{BI}}(\mathfrak{s})$ with $\{(\alpha_{1},...,\alpha_{k})|\alpha_{i}>0\}$.

\end{proof}

\begin{Rem}\label{semisimport}

From Theorem \ref{bitilde}, we can conclude that in a compact and semi-simple group $S$ with $\mathfrak{s}=\mathfrak{a}_{1}\oplus\cdot\cdot\cdot\oplus\mathfrak{a}_{k}$, the decomposition $\mathfrak{a}_{1}\oplus\cdot\cdot\cdot\oplus\mathfrak{a}_{k}$ is orthogonal with respect to any $\langle\cdot,\cdot\rangle\in\widetilde{\mathfrak{BI}}(\mathfrak{s})$.

\end{Rem}

\begin{Rem}\label{grbi}
    
Note that in this case, the space $\widetilde{\mathfrak{BI}}(\mathfrak{s})=\{(\alpha_{1},...,\alpha_{k})|\alpha_{i}>0\}$ has the structure of a group with the product of real numbers entry by entry. This fact has strong consequences later.

\end{Rem}

\subsection{Action of the automorphisms group on \texorpdfstring{$\widetilde{\mathfrak{BI}}$}{BI}.}

By Lemma \ref{adinbi} we recall that a bi-invariant metric on a Lie group $G$ is invariant under $Ad(g)$ for all $g\in G$. Thus any metric in $\widetilde{\mathfrak{BI}}(\mathfrak{g})$ remains fixed under the action $(Ad(g),\langle\cdot ,\cdot \rangle)\mapsto Ad(g)^{*}\langle\cdot ,\cdot \rangle=\langle\cdot,\cdot\rangle$. 
\\

This leads us to studying the action of the automorphisms of $\mathfrak{g}$ that are distinct from the ones given as $Ad(g)$. In order to give a complete description of the moduli spaces of bi-invariant metrics, we need to know more about the group of automorphisms of a semi-simple Lie algebra.

\begin{Def}

For $G$ a connected Lie group, we define the group of inner automorphisms of the Lie algebra $\mathfrak{g}$ as

$$Inn(\mathfrak{g})=\{Ad(g):\mathfrak{g}\rightarrow\mathfrak{g}|g\in G\}.$$

\end{Def}

It is easy to verify that $Inn(\mathfrak{g})$ does not depend on $G$. In \cite[Chapter II]{Serre comp} we can find the following result.

\begin{Thm}
Let $G$ be a connected semi-simple Lie group. Then the connected component of the identity of the automorphism group $Aut^{0}(G)$ in $Aut(G)$ coincides with $Inn(G).$
\end{Thm}

\begin{Rem}

The inner automorphisms of $\mathfrak{g}$ are induced by inner automorphisms of $G$, since for each $g\in G$, $Ad(g)$ is the differential of some inner automorphism of $G$. When $\mathfrak{g}$ is semi-simple, then $Inn(\mathfrak{g})=Aut^{0}(\mathfrak{g})$, that is, the connected component of the identity in $Aut(\mathfrak{g})$ is comprised solely of internal automorphisms.
    
\end{Rem}

In this way, we know that if $S$ is compact and semi-simple then for its Lie algebra $\mathfrak{s}$, the action of $Inn(\mathfrak{s})$ on $\widetilde{\mathfrak{BI}}(\mathfrak{s})$ is trivial, so we must only study how the outer automorphisms of the Lie algebra act on $\widetilde{\mathfrak{BI}}(\mathfrak{s})$.

\begin{Def}

For a Lie algebra $\mathfrak{g}$, we define the group of outer automorphisms as $Out(\mathfrak{g})=Aut(\mathfrak{g})/Inn(\mathfrak{g}).$
    
\end{Def}

We recall that for a Lie group $G$, the quotient $G/G^{0}$ is a group where each element corresponds to a connected component of $G$. In \cite[Corollary 2]{Murakami} it is verified that for a connected and semi-simple Lie group we have that $Aut(G)/Aut^{0}(G)$ is finite. In this way, we conclude that if $\mathfrak{s}$ is a semi-simple Lie algebra, then $Aut(\mathfrak{s})$ has a finite number of connected components and therefore $Out(\mathfrak{s})$ is a finite group.

 \begin{Lema}\label{compsimpinv}

If $G$ is a compact and simple Lie group then the action of $Out(\mathfrak{g})$ on $\widetilde{\mathfrak{BI}}(\mathfrak{g})$ is trivial.

\end{Lema}

\begin{proof}

Take the bi-invariant metric given by $-B$, where $B$ is the Killing form of $\mathfrak{g}$. Let $\phi\in Out(\mathfrak{g})$. As the metric $\phi^{*}(-B)$ is bi-invariant, then by Lemma \ref{bisim} $\phi^{*}(-B)=\alpha(-B)$ for some $\alpha>0$. Since $Out(\mathfrak{g})$ has finite order, then $\phi^{n}=Id$ for some $n\in\mathbb{N}$, and thus $(\phi^{n})^{*}(-B)=\alpha^{n}(-B)=-B$. This implies that $\alpha=1$ and therefore the action of $Out(\mathfrak{g})$ is trivial on $\widetilde{\mathfrak{BI}}(\mathfrak{g})$.
    
\end{proof}

\begin{Rem}

We have that for a compact and simple group $G$, the action of $Aut(\mathfrak{g})$ on $\widetilde{\mathfrak{BI}}(\mathfrak{g})$ is trivial. With this we conclude that $\mathfrak{BI}(\mathfrak{g})=\widetilde{\mathfrak{BI}}(\mathfrak{g})=\mathbb{R}^{+}$ and $\mathfrak{EBI}(\mathfrak{g})=\{\langle\cdot,\cdot\rangle_{0}\}.$
    
\end{Rem}

Now we study the semi-simple case. To do this, we review the following results.

\begin{Prop}

If $\mathfrak{u}$ is an ideal in a Lie algebra $\mathfrak{g}$ and $\phi\in Aut(\mathfrak{g})$, then $\phi(\mathfrak{u})$ is an ideal in $\mathfrak{g}.$

\end{Prop}

\begin{proof}

Let $u \in \mathfrak{u}$ and $x \in \mathfrak{g}$. We observe that $$[\phi(u),x]=[\phi(u),\phi(\phi^{-1}(x))]=\phi([u,\phi^{-1}(x)]) \in \phi(\mathfrak{u}).$$

Thus $\phi(\mathfrak{u})$ is an ideal in $\mathfrak{g}$.

\end{proof}

\begin{Coro}\label{semij}

Let $\mathfrak{s}=\mathfrak{a}_{1}\oplus\cdots\oplus\mathfrak{a}_{k}$ be a semi-simple Lie algebra and $\phi\in Aut(\mathfrak{g})$. Then $\phi(\mathfrak{a}_{i})=\mathfrak{a}_{j}$ for some $j.$
    
\end{Coro}

\begin{proof}

We know that $\phi(\mathfrak{a}_{i})$ is a simple ideal of $\mathfrak{g}$. Furthermore, for all $j$, $\phi(\mathfrak{a}_{i})\cap\mathfrak{a}_{j}$ is an ideal of $\mathfrak{a}_{j}$. But since $\mathfrak{a}_{j}$ is simple, then $\phi(\mathfrak{a}_{i})\cap\mathfrak{a}_{j}={0}$ or $\phi(\mathfrak{a}_{i})\cap\mathfrak{a}_{j}=\mathfrak{a}_{j}$. In the second case, we have that $\mathfrak{a}_{j}\subset\phi(\mathfrak{a}_{i})$ but since $\phi(\mathfrak{a}_{i})$ is simple, then $\phi(\mathfrak{a}_{i})=\mathfrak{a}_{j}$. Also, if $\phi(\mathfrak{a}_{i})\cap\mathfrak{a}_{j}={0}$ for all $j$, then $\phi(\mathfrak{a}_{i})=0$ which is a contradiction.

\end{proof}

The next step will be to study the semi-simple case where all the factors are non-isomorphic to each other.

\begin{Lema}

If $\mathfrak{s}=\mathfrak{a}_{1}\oplus\cdots\oplus\mathfrak{a}_{k}$ is a semi-simple Lie algebra where the summands are simple and not isomorphic to each other, then $Aut(\mathfrak{s})\cong Aut(\mathfrak{a}_{1})\oplus\cdots\oplus Aut(\mathfrak{a}_{k})$.
    
\end{Lema}

\begin{proof}

Let $\phi\in Aut(\mathfrak{s})$. Since all the summands are not isomorphic to each other by the previous corollary we have that $\phi(\mathfrak{a}_{i})=\mathfrak{a}_{i}$. Thus we conclude that $\phi=\phi_{1}+\cdots+\phi_{k}$, where each $\phi_{i}\in Aut(\mathfrak{a}_{i})$, which finishes the proof.
    
\end{proof}

\begin{Lema}

Let $G$ be a compact, connected, and semi-simple Lie group with Lie algebra $\mathfrak{g}=\mathfrak{a}_{1}\oplus\cdots\oplus\mathfrak{a}_{k}$, where the summands are not isomorphic to each other and are simple. Then the action of $Out(\mathfrak{g})$ on $\widetilde{\mathfrak{BI}}(\mathfrak{g})$ is trivial.
    
\end{Lema}

\begin{proof}

Let $\langle\cdot ,\cdot \rangle_{0}=(-B_{1})+\cdots+(-B_{k})$ be the bi-invariant metric given by the sum of the Killing forms $B_{i}$ in each $\mathfrak{a}_{i}$. We define the mapping $$\Phi:Out(\mathfrak{g})\rightarrow\widetilde{\mathfrak{BI}}(\mathfrak{g}), \ \Phi(\psi)=\psi^{*}(\langle\cdot,\cdot\rangle_{0}).$$

From the previous lemma, we know that $\psi=\psi_{1}+\cdots+\psi_{k}$ where $\psi_{i}\in Aut(\mathfrak{a}_{i})$. Furthermore, each $\psi_{i}^{*}(-B_{i})=\alpha_{i}(-B_{i})$ with $\alpha_{i}>0$ because it is a bi-invariant metric. In this way, we conclude that $$\Phi(\psi)=\psi^{*}(\langle\cdot,\cdot\rangle_{0})=\psi_{1}^{*}(-B_{1})+\cdots+\psi_{k}^{*}(-B_{k})=\alpha_{1}(-B_{1})+\cdots+\alpha_{k}(-B_{k}).$$ For another $\mu\in Out(\mathfrak{g})$, we have $\Phi(\mu)=\beta_{1}(-B_{1})+\cdots+\beta_{k}(-B_{k})$ and we conclude that $\Phi(\mu\circ\psi)=\alpha_{1}\beta_{1}(-B_{1})+\cdots+\alpha_{k}\beta_{k}(-B_{k})$. In Remark \ref{grbi} we had identified $\widetilde{\mathfrak{BI}}(\mathfrak{g})$ with the multiplicative group of vectors with $k$ positive real entries. From what we have seen above, we conclude that $\Phi$ is a group homomorphism and the image of $Out(\mathfrak{g})$ is a finite subgroup of $\widetilde{\mathfrak{BI}}(\mathfrak{g})$. But this group does not have any non-trivial finite subgroups, so we conclude that $\Phi$ is trivial and the action of $Out(\mathfrak{g})$ on $\widetilde{\mathfrak{BI}}(\mathfrak{g})$ is also trivial.
    
\end{proof}

\subsection{Description of \texorpdfstring{$\mathfrak{BI}$}{BI} and \texorpdfstring{$\mathfrak{EBI}$}{EBI}}

With what has been developed so far, we can state the following theorem.

\begin{Thm}

For a compact, connected, and semi-simple Lie group $G$ with $\mathfrak{g}=\mathfrak{a}_{1}\oplus\cdot\cdot\cdot\oplus\mathfrak{a}_{k}$ where the summands are simple and pairwise
non-isomorphic, it follows that $\mathfrak{BI}(\mathfrak{g})=\widetilde{\mathfrak{BI}}(\mathfrak{g})$. Furthermore, $\mathfrak{EBI}(\mathfrak{g})=\widetilde{\mathfrak{BI}}(\mathfrak{g})/\mathfrak{\mathbb{R}^{\times}}$ corresponds to the set of vectors with $k-1$ strictly positive real number entries.
    
\end{Thm}

Now we are going to study the case $G=H\times\cdot\cdot\cdot\times H$, where $H$ is a compact and simple Lie group, with Lie algebra $\mathfrak{g}=\mathfrak{h}_{1}\oplus\cdot\cdot\cdot\oplus\mathfrak{h}_{k}$, and each $\mathfrak{h}_{i}\cong Lie(H)$. For an automorphism $\phi\in Aut(\mathfrak{h}_{1}\oplus\cdot\cdot\cdot\oplus\mathfrak{h}_{k})$ Corollary \ref{semij} ensures that $\phi(\mathfrak{h}_{i})=\mathfrak{h}_{j}$. This small detail allows us to relate $Aut(\mathfrak{g})$ to the symmetric group of permutations. We start with the following definition.

\begin{Def}

Given a set $X$ and $S_{n}$ the symmetric group of permutations of $n$ elements, we define the left action of $S_{n}$ on $X^{n}=X\times\cdot\cdot\cdot\times X$ given by $$(\sigma,(x_{1},...,x_{n}))\mapsto(x_{\sigma(1)},...,x_{\sigma(n)}). $$
    
\end{Def}

\begin{Rem}

For $\mathfrak{g}=\mathfrak{h}_{1}\oplus\cdot\cdot\cdot\oplus\mathfrak{h}_{k}$ semi-simple, where $\mathfrak{h}_{i}\cong\mathfrak{a}$, we have a group homomorphism $T:Aut(\mathfrak{g})\rightarrow S_{k}$ given in the following way: if $\phi\in Aut(\mathfrak{g})$, then $T(\phi)\in S_{k}$ is such that $T(\phi)(i)=j$ if and only if $\phi(\mathfrak{h}_{i})=\mathfrak{h}_{j}.$ In conclusion, we have that each automorphism of $\mathfrak{h}_{1}\oplus\cdot\cdot\cdot\oplus\mathfrak{h}_{k}$ is given by automorphims ${\phi_{i}\in Aut(\mathfrak{h}_{i})}$ for $i=1,\ldots,k$, and a permutation $\sigma\in S_{k},$ where it holds that, if $x_{i}\in\mathfrak{h}_{i}$, then $\phi_{i}(x_{i})\in\mathfrak{h}_{\sigma(i)}$.
    
\end{Rem}

Now we will see what happens when each $\mathfrak{h}_{i}$ is the Lie algebra of a simple and compact group.

\begin{Thm}

Let $H$ be a connected, simple, and compact Lie group and consider the product of $H$ $k$-times, $G=H\times\cdot\cdot\cdot\times H$. Then the action of $Aut(\mathfrak{g})$ on $\widetilde{\mathfrak{BI}}(\mathfrak{g})$ coincides with the action of the symmetric group $S_{k}$. Observe that the inner  product $\langle\cdot,\cdot\rangle$ can be decomposed as a  sum of inner products:
\[
\langle\cdot,\cdot\rangle = \langle\cdot,\cdot\rangle_{1}+\ldots +\langle\cdot,\cdot\rangle_{k},
\]
where $\langle\cdot,\cdot\rangle_i = \langle\cdot,\cdot\rangle|_{\mathfrak{h}_i}$. With this we set
\[
(\sigma, \langle\cdot ,\cdot\rangle) \mapsto  \langle\cdot,\cdot\rangle_{\sigma(1)}+\ldots +\langle\cdot,\cdot\rangle_{\sigma(k)}.
\]
    
\end{Thm}

\begin{proof}

Let $\phi\in Aut(\mathfrak{g})$ and let us see how this automorphism acts on a bi-invariant metric $\langle\cdot ,\cdot \rangle_{0}\in\widetilde{\mathfrak{BI}}(\mathfrak{g})$. We know that $\langle \cdot,\cdot \rangle_{0}=\alpha_{1}(-B)+\cdot\cdot\cdot+\alpha_{k}(-B)$, where $B$ is the Killing form on $H$ and each $\alpha_{i}>0$. Using Observation \ref{semisimport} we know that the decomposition $\mathfrak{h}\oplus\cdot\cdot\cdot\oplus\mathfrak{h}$ is orthogonal for any bi-invariant metric. From this we obtain

$$
\begin{array}{ll}
\phi^{*}(\langle \cdot,\cdot\rangle_{0})(\sum_{i}^{k}x_{i},\sum_{i}^{k}y_{i})&=\sum_{i}^{k}\langle \phi_{i}(x_{i}),\phi_{i}(y_{i})\rangle _{0}\\
             &=\sum_{i}^{k}\alpha_{\sigma(i)}(-B)(\phi_{i}(x_{i}),\phi_{i}(y_{i})).
\end{array}
$$

But as we saw in Lemma \ref{compsimpinv} the action of the automorphism group of a compact and simple group on the space $\widetilde{\mathfrak{BI}}(\mathfrak{h})$ is trivial. With this we have that

$$
\begin{array}{ll}
  \phi^{*}(\langle \cdot,\cdot\rangle_{0})(\sum_{i}^{k}x_{i},\sum_{i}^{k}y_{i})&=\sum_{i}^{k}\alpha_{\sigma(i)}(-B)(x_{i},y_{i})\\
 &=\sum_{i}^{k}\langle x_{\sigma(i)},  y_{\sigma(i)}\rangle_{0}.
\end{array}
$$

Thus we have that the action of $Aut(\mathfrak{g})$ on $\widetilde{\mathfrak{BI}}(\mathfrak{g})$ is given by the action of the symmetric group $S_{k}.$
    
\end{proof}

\begin{Def}

We consider the action of the symmetric group of permutations $S_{n}$ on $X^{n}=X\times\cdots\times X$, where $X$ is a topological space, given by $$(\sigma,(x_{1},...,x_{n}))\mapsto(x_{\sigma(1)},...,x_{\sigma(n)}).$$ We define the $n$th symmetric product of $X$ as $SP^{n}(X)=X^{n}/S_{n}.$
    
\end{Def}

By \cite{Adem}, when $M$ is a differentiable manifold, then $SP^{n}(M)$ has an orbifold structure, and if $M=\mathbb{R}$ then $SP^{n}(\mathbb{R})$ is homeomorphic to the product $\mathbb{R}\times(\mathbb{R}^{+}\cup\{0\})^{n-1}.$

\begin{Thm}

If $H$ is a compact, connected, simple Lie group and $G$ is the product of $H$ $k$-copies of $H$, i.e. $G=H\times\cdots\times H$, then the moduli space $\mathfrak{BI}(\mathfrak{g})$ is homeomorphic to the $k$th symmetric product of $\mathbb{R}$, $SP^{k}(\mathbb{R})\cong \mathbb{R}\times(\mathbb{R}^{+}\cup{0})^{k-1}$.
    
\end{Thm}

\begin{proof}

Let $\langle\cdot,\cdot\rangle=\alpha_{1}(-B)+...+\alpha_{k}(-B)\in\widetilde{\mathfrak{BI}}(\mathfrak{g})$, $x=\sum _{i}^{k}x_{i},y=\sum_{i}^{k}y_{i}\in\mathfrak{g}$ and $\sigma\in S_{k}$. We compute:

$$
\begin{array}{ll}
\langle \sum_{i}^{k} x_{\sigma(i)},\sum_{i}^{k}y_{\sigma(i)} \rangle&=\sum_{i}^{k}\alpha_{i}(-B(x_{\sigma(i)},y_{\sigma(i)}))\\
&=\sum_{i}^{k}\alpha_{\sigma(i)}(-B)(x_{i},y_{i}).
\end{array}
$$

With this we conclude that the isometry class of a bi-invariant metric is the set $[(\alpha_{1},...,\alpha_{k})]=\{(\alpha_{\sigma(1)},...,\alpha_{\sigma(k)})|\sigma\in S_{k} \}$. It follows that $$\mathfrak{BI}(\mathfrak{g})\cong(\mathbb{R}^{+})^{k}/S_{k}\cong\mathbb{R}^{k}/S_{k}=SP^{k}(\mathbb{R})\cong\mathbb{R}\times(\mathbb{R}^{+}\cup\{0\})^{k-1}.$$

\end{proof}

\begin{Rem}

In this case, $\mathfrak{BJ}(\mathfrak{g})$ is homeomorphic to $\mathbb{R}\times(\mathbb{R}^{+}\cup\{0\})^{k-1}$, which is homotopic to a point and thus contractible.
    
\end{Rem}

\begin{Rem}

In the case of $\mathfrak{EBI}(\mathfrak{g})$, recall that by Remark \ref{acom} we have that the actions of $Aut(\mathfrak{g})$ and $\mathbb{R}^{\times}$ on $\widetilde{\mathfrak{BI}}(\mathfrak{g})$ commute, so we can first consider the space $\widetilde{\mathfrak{BI}}(\mathfrak{g})/\mathbb{R}^{\times}$ and later factor it with $Aut(\mathfrak{g})$. Since $\widetilde{\mathfrak{BI}}(\mathfrak{g})$ is formed by vectors of strictly positive $k$ real numbers, we can think of the space $\widetilde{\mathfrak{BI}}(\mathfrak{g})/\mathbb{R}^{\times}$ as $\mathbb{S}^{k-1}_{+}=\mathbb{S}^{k-1}\cap\widetilde{\mathfrak{BI}}(\mathfrak{g})$, and in this way $\mathfrak{EBI}(\mathfrak{g})=\mathbb{S}^{k-1}_{+}/Aut(\mathfrak{g})$.
    
\end{Rem}

Having said this, we can state the following corollary.

\begin{Coro}
If $H$ is a compact, connected, simple Lie group, and $G$ is the product of $H$ $k$-copies of $H$ i.e. $G=H\times\cdots\times H$, then $\mathfrak{EBI}(\mathfrak{g})=\mathbb{S}^{k-1}_{+}/S_{k}$.

\end{Coro}

\begin{Rem}

We consider the homotopy $H:\mathbb{S}^{k-1}_{+}\times I\rightarrow\mathbb{S}^{k-1}_{+}$ given by $$H((x_{1},...,x_{k}),t)=\frac{(1-t)(x_{1},...,x_{k})+t(\frac{1}{\sqrt{k}})(1,...,1)}{||(1-t)(x_{1},...,x_{k})+t(\frac{1}{\sqrt{k}})(1,...,1) ||},$$  where $H(x,0)=x$ and $H(x,1)=\left(\frac{1}{\sqrt{k}}\right)(1,...,1)\in\mathbb{S}^{k-1}_{+}$, then $\mathbb{S}^{k-1}_{+}$ is homotopic to the point $\left(\frac{1}{\sqrt{k}}\right)(1,...,1).$ Now, we define $\tilde{H}:(\mathbb{S}^{k-1}_{+}/S{k})\times I\rightarrow(\mathbb{S}^{k-1}_{+}/S{k})$ given by $\tilde{H}([x_{1},...,x_{k}],t)=[H((x_{1},...,x_{k}),t)]$. It is easy to see that $\tilde{H}$ is well-defined and is a retraction from $(\mathbb{S}^{k-1}_{+}/S^{k})$ to $[(\frac{1}{\sqrt{k}})(1,...,1)]$. Therefore, $\mathfrak{EBI}(\mathfrak{g})$ is contractible.
    
\end{Rem}

\subsection{General Case}

\begin{Thm}

Let $G$ be a connected, compact, and semi-simple Lie group with $\mathfrak{g}\cong\mathfrak{s}\oplus\mathfrak{b}_{1}\oplus\cdots\oplus\mathfrak{b}_{l}$ where $\mathfrak{s}\cong\mathfrak{a}_{1}\oplus\cdots\oplus\mathfrak{a}_{k}$ is semi-simple with non-isomorphic factors, and each $\mathfrak{b}_{i}$ is semi-simple and decomposes as the sum of $m_{i}$ isomorphic factors but none of these factors is isomorphic to any $\mathfrak{a}_{j}$. Then $\mathfrak{BI}(\mathfrak{g})$ is homeomorphic to
$$\widetilde{\mathfrak{BI}}(\mathfrak{s})\times SP^{m_{1}}(\mathbb{R})\times\cdots\times SP^{m_{l}}(\mathbb{R}).$$
    
\end{Thm}

\begin{proof}
Let $\phi \in Aut(\mathfrak{g})$. By Corollary \ref{semij}, we have $\phi=\psi+\varphi_{1}+\cdots+\varphi_{l}$ where $\psi \in Aut(\mathfrak{s})$ and each $\varphi_{i} \in Aut(\mathfrak{b}_{i})$. We know that the action of $\psi$ on $\widetilde{\mathfrak{BI}}(\mathfrak{s})$ is trivial and the action of $\varphi_{i}$ on $\widetilde{\mathfrak{BI}}(\mathfrak{b}_{i})$ is through a permutation $\sigma \in S_{m_{i}}$. Therefore, upon factoring $\widetilde{\mathfrak{BI}}(\mathfrak{g})$ with $Aut(\mathfrak{g})$, we have that $\mathfrak{BI}(\mathfrak{g})$ is homeomorphic to the product $\widetilde{\mathfrak{BI}}(\mathfrak{s})\times SP^{m_{1}}(\mathbb{R})\times\cdots\times SP^{m_{l}}(\mathbb{R})$.
\end{proof}

And for the space $\mathfrak{EBI}(\mathfrak{g})$ we have the following theorem.

\begin{Thm}

Let $G$ be a connected, compact, and semi-simple Lie group with $\mathfrak{g}\cong\mathfrak{s}\oplus\mathfrak{b}_{1}\oplus\cdots\oplus\mathfrak{b}_{l}$, where $\mathfrak{s}\cong\mathfrak{a}_{1}\oplus\cdots\oplus\mathfrak{a}_{k}$ is semi-simple with non-isomorphic factors and each $\mathfrak{b}_{i}$ is semi-simple and decomposes into a sum of $m_{i}$ isomorphic factors, but none of these factors are isomorphic to any $\mathfrak{a}_{j}$. Then $\mathfrak{EBI}(\mathfrak{g})$ is homeomorphic to $$\mathfrak{EBI}(\mathfrak{s})\times(\mathbb{S}^{m{1}-1}_{+}/S_{m_{1}})\times\cdots\times(\mathbb{S}^{m_{l}-1}_{+}/S_{m_{l}}).$$
    
\end{Thm}

In both cases, the spaces $\mathfrak{BI}(\mathfrak{g})$ and $\mathfrak{EBI}(\mathfrak{g})$ are products of contractible spaces and therefore they are contractible as well. Also note that for a Lie group $G$  admitting a bi-invariant metric, the description of  these spaces depends on the decomposition of the Lie algebra into semi-simple components, and in particular on the number of simple components which are isomorphic.

\section{Examples}\label{Section 4}

In this section we compute explicitly the moduli space $\mathfrak{BI}$ and $\mathfrak{EBI}$ for the semi-simple Lie groups of dimension at most 6. In dimension $1$, all Lie groups are abelian. In dimension $2$, there is only one non-abelian Lie algebra $\mathfrak{g}^{2}$ which has a basis $e_{1},e_{2}\in\mathfrak{g}^{2}$ where the Lie bracket is given by $[e_{1},e_{2}]=e_{2}$. It does not admit a bi-invariant metric because it does not have simple ideals. The first example of bi-invariant metrics in a non-abelian group appears in dimension $3$ with $SU(2)$.

\subsection{\texorpdfstring{$\mathfrak{BI}(\mathfrak{su}(2))$}{BI(su(2))} and \texorpdfstring{$\mathfrak{EBI}(\mathfrak{su}(2))$}{EBI(su(2))}}\label{round} In a three-dimensional Lie algebra $\mathfrak{g}^{3}$ that has a bi-invariant metric with an orthonormal basis $e_{1},e_{2},e_{3}\in\mathfrak{g}^{3}$, the structure constants satisfy $\alpha_{ijk}=-\alpha_{ikj}$. Therefore, it is easy to conclude that $[e_{1},e_{2}]=\lambda e_{3}$, $[e_{2},e_{3}]=\lambda e_{1}$, $[e_{3},e_{1}]=\lambda e_{2}$, where $\lambda>0$. We will verify that this Lie algebra is isomorphic to $\mathfrak{su}(2)$. Let $a_{1},a_{2},a_{3}\in\mathfrak{su}(2)$ be the basis that satisfies $[a_{1},a_{2}]=a_{3}$, $[a_{2},a_{3}]=a_{1}$, and $[a_{3},a_{1}]=a_{2}$. We define $\phi:\mathfrak{su}(2)\rightarrow\mathfrak{g}^{3}$ as $\phi(a_{i})=\frac{1}{\lambda}e_{i}$. Then, $\phi$ is an isomorphism of Lie algebras because $\phi([a_{i},a_{j}])=\phi(a_{k})=\frac{1}{\lambda}e_{k}$ and $[\phi(a_{i}),\phi(a_{j})]=[\frac{1}{\lambda}e_{i},\frac{1}{\lambda}e_{j}]=\lambda(\frac{1}{\lambda^{2}})e_{k}=\frac{1}{\lambda}e_{k}$. This tells us that there is only one Lie algebra with a bi-invariant metric in dimension $3$. Since $SU(2)$ is compact and simple, then $\mathfrak{BI}(\mathfrak{su}(2))=\mathbb{R}^{+}$ and $\mathfrak{EBI}(\mathfrak{su}(2))={\langle\cdot,\cdot\rangle_{0}}$.
\\

 In dimensions $4$ and $5$, all Lie algebras admitting a bi-invariant metric are either abelian or have $\mathfrak{su}(2)$ as a factor. This is because there are no semi-simple Lie algebras in dimensions $4$ and $5$. Using \ref{autbi}, we conclude that the moduli spaces $\mathfrak{BI}$ and $\mathfrak{EBI}$ are homeomorphic to $\mathfrak{BI}(\mathfrak{su}(2))$ and $\mathfrak{EBI}(\mathfrak{su}(2))$ respectively.

\subsection{\texorpdfstring{$\mathfrak{BI}(\mathfrak{so}(4))$}{BI(so(4))} and \texorpdfstring{$\mathfrak{EBI}(\mathfrak{so}(4))$}{EBI(so(4))}}

The Lie group $SO(4)$ has dimension $6$. Moreover, it is compact and semi-simple, since it satisfies $\mathfrak{so}(4)\cong\mathfrak{su}(2)\oplus\mathfrak{su}(2)$. With two summands, we have that $\widetilde{\mathfrak{BI}}(\mathfrak{so}(4))=(\mathbb{R}^{+})^{2}$. Thus, the moduli space $\mathfrak{BI}(\mathfrak{so}(4))\cong(\mathbb{R}^{+})^{2}/S_{2}$ is homeomorphic to $\mathbb{R}^{+}\times\mathbb{R}$, and $\mathfrak{EBI}(\mathfrak{so}(4))$ is homeomorphic to $\mathbb{R}^{+}$ (see figure \ref{fig:etiqueta}) This implies that
up to rescalling there is only a
one parameter group of bi-invariant
metrics on $SO(4).$ 

\begin{figure}[h]
  \centering
  \includegraphics[scale=0.12]{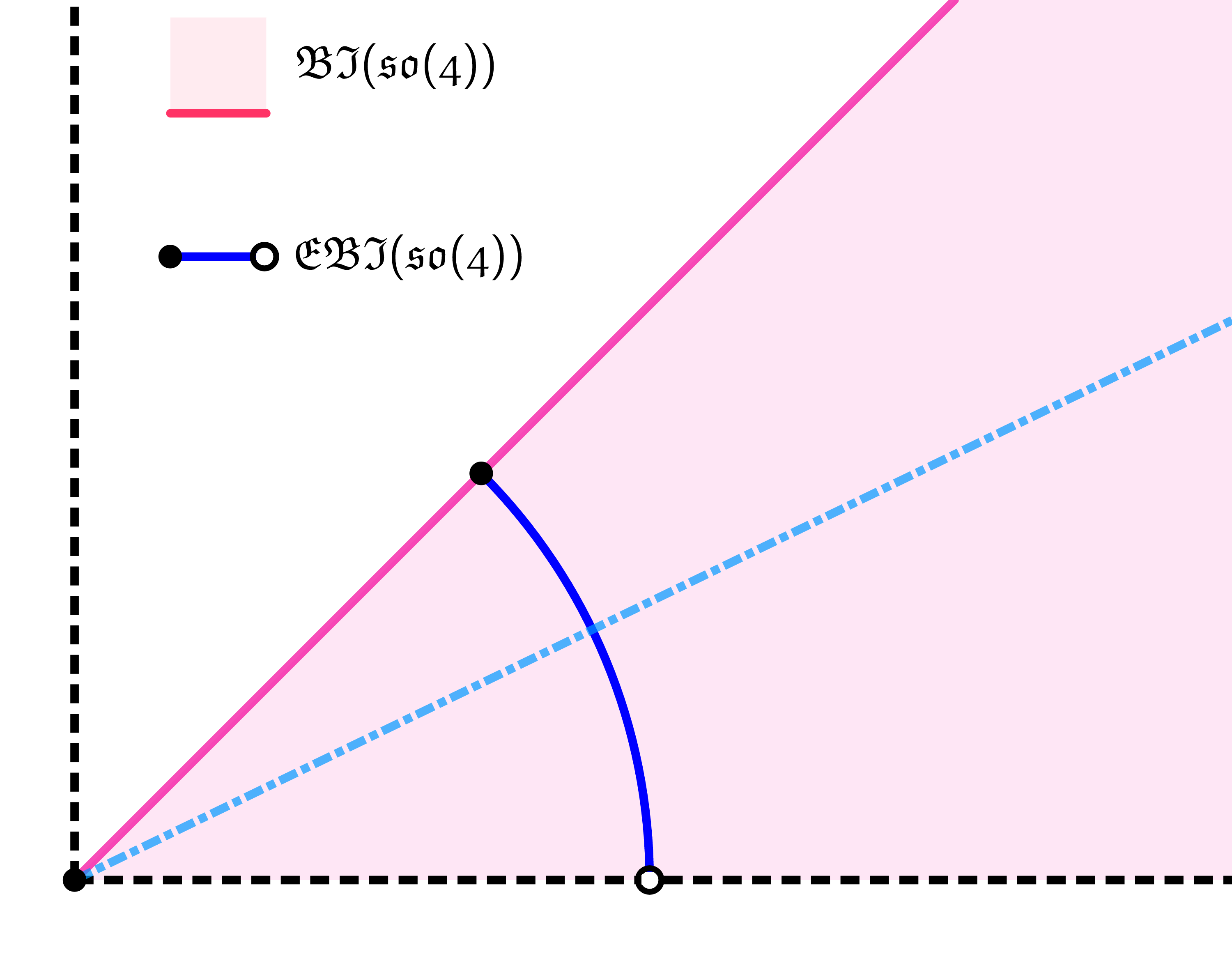}
  \caption{Since $\widetilde{\mathfrak{BI}}(\mathfrak{so}(4))=(\mathbb{R}^{+})^{2}$, each metric is represented by a pair of positive real numbers. Therefore, we can visualize the space $\widetilde{\mathfrak{BI}}(\mathfrak{so}(4))$ as the positive quadrant of $\mathbb{R}^{2}$, where each point $(\alpha_{1},\alpha_{2})$ represents a bi-invariant metric. Taking the quotient of $\widetilde{\mathfrak{BI}}(\mathfrak{so}(4))$ by $S_{2}$, each $(\alpha_{1},\alpha_{2})$ is identified with $(\alpha_{2},\alpha_{1})$. This implies that $\mathfrak{BI}(\mathfrak{so}(4))$ is the shaded area with the boundary given by the line $y=x$. The dotted line passing through the shaded area represents an element of $\mathfrak{EBI}(\mathfrak{so}(4))$, so this space is represented by the circular segment extending from the $x$-axis to the line $x=y$.}
  \label{fig:etiqueta}
\end{figure}


\section{Curvature}\label{Section 5}

Using the language developed so far, we can restate some results from Milnor's previous work in \cite{Milnor} as follows.

\subsection{Strictly Positive Sectional Curvature}

For metrics $\langle\cdot,\cdot\rangle\in\mathfrak{BI}(\mathfrak{g})$ with $K>0$, Wallach's Theorem \cite[Theorem 2.1]{Wallach} will be of great utility.

\begin{Wthm}\label{Wthm}

 The $3$-shpere group $SU(2)$ is the only simply connected Lie group wich admits a left invariant metric of strictly positive sectional curvature.
    
\end{Wthm}

Consequently, we have the following corollary.

\begin{Coro}\label{curvsecjuan}

If $G$ is a Lie group with a bi-invariant metric such that $K > 0$, then $G$ is compact and has dimension equal to three.

\end{Coro}

\begin{proof}

Since $G$ admits a bi-invariant metric, then $\widetilde{G}$ does as well. According to Lemma \ref{bi5}, we have $\widetilde{G}\cong H\times\mathbb{R}^{l}$ where $H$ is a compact group. If $v\in\mathfrak{g}$ is unitary and tangent to $\mathbb{R}^{l}$, it satisfies $[v,u]=0$ for every unitary $u$. Consequently by Theorem \ref{curvsecbicor}, $K(v,u)=\frac{\kappa(v,u)}{1-\langle v,u\rangle}=\frac{\frac{1}{4}||[u,v]||^{2}}{1-\langle v,u \rangle}=0$. This implies that $\widetilde{G}$ must be compact. Additionally, being simply connected allows us to apply Wallach's Theorem \ref{Wthm}, leading to the conclusion that $\widetilde{G}\cong SU(2)$. Therefore, $G\cong SU(2)/L$ where $L$ is a finite subgroup of $SU(2)$.
    
\end{proof}

\begin{Rem}

In conclusion, we have that, If $\langle\cdot,\cdot\rangle\in\mathfrak{BI}(\mathfrak{g})$ is such that $K>0$, then $\mathfrak{g}\cong\mathfrak{su}(2)$  
 \ref{curvsecjuan}. Thus by \ref{round} there exist up to rescalling only one bi-invariant Riemannian geometry. Observe that is the geometry of the round 3-sphere (see \cite{spheres}).
    
\end{Rem}

\subsection{Non-negative Sectional Curvature}
Suppose $G$ is the semidirect product $P\rtimes H$ of a subgroup $H$ with a bi-invariant metric and a normal, commutative subgroup $P$ with a left-invariant and $Ad(h)$-invariant metric for all $h\in H$. Then any $\langle\cdot,\cdot\rangle\in\mathfrak{BI}(\mathfrak{h})$ can be extended to a metric in $\mathfrak{M}(\mathfrak{g})$ such that $K\geq0$, \cite[p. 297]{Milnor}. Milnor conjectures that these are the only groups that admit such a metric.

\subsection{Flat bi-invariant metrics}
 If $\langle\cdot,\cdot\rangle\in\mathfrak{BI}(\mathfrak{g})$ is such that $K=0$, then any element of $\mathfrak{BI}(\mathfrak{g})$ has zero sectional curvature \cite[Theorem 1.5]{Milnor}.

\subsection{Ricci Curvature}

\begin{enumerate}

\item If $G$ is compact and semi-simple, then any $\langle\cdot,\cdot\rangle\in\mathfrak{BI}(\mathfrak{g})$ satisfies $Ric>0$, and $\pi_{1}(G)$ is finite, see \cite[Theorem 2.2]{Milnor}.

\item If the universal cover $\widetilde{G}$ of $G$ is compact, then there exist a bi-invariant metric in $\mathfrak{BI}(\mathfrak{g})$ wich has constant Ricci curvature \cite[Corollary 7.7]{Milnor}. Moreover each element in $\mathfrak{EBI}(\mathfrak{g})$ has a representative with $Ric=1.$

\end{enumerate}

\subsection{Scalar Curvature} 

If a Lie group $G$ has a non-commutative compact subgroup $H$, then any element of $\mathfrak{BI}(\mathfrak{h})$ extends to a metric in $\mathfrak{M}(\mathfrak{g})$ with $\rho>0$, see \cite[Theorem 3.4]{Milnor}.

\section*{Funding and Conflict of Interest}

This work was financed by a Consejo Nacional de la Ciencia y Tecnología  (CONACyT) Master degree Scholarship No: 848752. The author declares that there is no conflict of interest. Data sharing is not applicable to this article as no data sets were generated or analysed for this work.


\begin{thebibliography}{20}

\bibitem{Adem} \textsc{Adem, Alejandro; Leida, Johan; Ruan, Yongbin}.
\textit{Orbifolds and Stringy Topology}. Cambridge University Press, Cambridge, 2007.

\bibitem{Alexandrino} \textsc{Alexandrino, Marcos; G. Bettiol, Renato}.
\textit{Lie Groups and Geometric Aspects of Isometric Actions}. Springer International Publishing, Switzerland 2015.

\bibitem{Cheeger} \textsc{Cheeger, Jeff}. \textit{Some examples of manifolds of nonnegative curvature}.
J. Differential Geometry \textbf{8}, (1973) 623–628. 

\bibitem{GroveZiller} \textsc{Grove, Karsten; Ziller, Wolfgang}. \textit{Cohomogeneity one manifolds with positive Ricci curvature}.
Inventiones Mathematicae \textbf{149}, no. 3, (2002) 619–646. 

\bibitem{spheres} \textsc{Kerin, M. and Wraith, D}.
\textit{Homogeneous Metrics on Spheres}. Bulletin of the Irish Mathematical Society \textbf{51}, (2003), 59–71. 

\bibitem{Kodama} \textsc{Kodama, Hiroshi; Takahara, Atsushi; Tamaru, Hiroshi}.
\textit{The space of left-invariant metrics on a Lie group up to isometry and scaling}. Manuscripta Math. \textbf{135}, no. 1-2, (2011) 229--243. 

\bibitem{Milnor} \textsc{Milnor, John}.
\textit{Curvatures of Left Invariant Metrics on Lie Groups}. Advances in Mathematics \textbf{21}, (1976) 293-329.

\bibitem{Murakami} \textsc{Murakami, Shingo}.
\textit{On the automorphism of a real semi-simple Lie algebra}. Journal of the Mathematical Society of Japan \textbf{4}, no. 2, (1952) 103--133.

\bibitem{Searle} \textsc{Searle, Catherine}.
\textit{Symmetries of spaces with lower curvature bounds}.
Notices American Mathematical Society \textbf{70}, no. 4, (2023) 564–575. 

\bibitem{SearleWilhelm} \textsc{Searle, Catherine; Wilhelm, Frederick}. \textit{How to lift positive Ricci curvature}. Geometry \& Topology \textbf{19}, no. 3, (2015) 1409–1475. 

\bibitem{Serre} \textsc{Serre, Jean-Pierre}.
\textit{Lie Algebras and Lie Groups}, Springer-Verlag Berlin Heidelberg 1992.

\bibitem{Serre comp} \textsc{Serre, Jean-Pierre}.
\textit{Complex Semisimple Lie Algebras}, Springer-Verlag Berlin Heidelberg 2001.

\bibitem{TuschmannWraith} \textsc{Tuschmann, Wilderich and Wraith, David J}. \textit{Moduli spaces of Riemannian metrics}. Second corrected printing. Oberwolfach Seminars, 46. Birkhäuser Verlag, Basel, 2015.

\bibitem{Wallach} \textsc{Wallach, Nolan R}. \textit{Compact homogeneous Riemannian manifolds of stricktly positive curvature}. Annals of Mathematics \textbf{96}, no. 2,
(1972) 277–295.


\end{thebibliography}
\end{document}